\documentclass[11pt]{amsart}

\usepackage{epsfig, color}

\usepackage{amsmath}
\usepackage{amsthm}
\usepackage{hyperref}
\usepackage[margin=1.0in]{geometry}

\numberwithin{equation}{section}

\newtheorem{prop}{Proposition}
\newtheorem{lemma}[prop]{Lemma}

\newtheorem{thm}[prop]{Theorem}

\numberwithin{prop}{section}

\theoremstyle{definition}
\newtheorem{defn}[prop]{Definition}

\newtheorem{rmk}[prop]{Remark}

\newcommand{\del}{\partial}
\newcommand{\delb}{\bar{\partial}}\newcommand{\dt}{\frac{\partial}{\partial t}}
\newcommand{\brs}[1]{\left| #1 \right|}

\renewcommand{\gg}{\gamma}
\newcommand{\gD}{\Delta}
\newcommand{\gd}{\delta}
\newcommand{\gs}{\sigma}

\newcommand{\gL}{\Lambda}

\newcommand{\gl}{\lambda}

\newcommand{\gw}{\omega}
\newcommand{\ga}{\alpha}
\newcommand{\gb}{\beta}
\renewcommand{\ge}{\epsilon}

\newcommand{\PP}{\mathcal P}

\newcommand{\til}[1]{\widetilde{#1}}

\renewcommand{\bar}[1]{\overline{#1}}

\renewcommand{\i}{\sqrt{-1}}

\newcommand{\ubar}[1]{\underline{#1}}

\newcommand{\bj}{\bar{j}}

\newcommand{\bw}{\bar{w}}

\newcommand{\bz}{\bar{z}}

\newcommand{\HH}{\mathcal{H}}

\DeclareMathOperator{\Rc}{Rc}

\DeclareMathOperator{\Id}{Id}

\DeclareMathOperator{\End}{End}
\DeclareMathOperator{\USC}{USC}
\DeclareMathOperator{\LSC}{LSC}

\begin{document}

\title[Global viscosity solutions of generalized K\"ahler-Ricci flow]{Global
viscosity solutions of generalized K\"ahler-Ricci flow}

\begin{abstract} We apply ideas from viscosity theory to establish the existence
of a unique global weak solution to the generalized K\"ahler-Ricci flow in the
setting of commuting complex structures.  Our results are restricted to the case
of a smooth manifold with smooth background data.  We discuss the possibility of
extending these results to more singular settings, pointing out a key error in
the existing literature on viscosity solutions to complex Monge-Ampere equations/K\"ahler-Ricci flow.
\end{abstract}

\author{Jeffrey Streets}
\address{Rowland Hall\\
         University of California, Irvine\\
         Irvine, CA 92617}
\email{\href{mailto:jstreets@uci.edu}{jstreets@uci.edu}}
\thanks{J. Streets gratefully acknowledges support from the NSF via
DMS-1454854, and from an Alfred P. Sloan Fellowship.}

\date{October 6th, 2016}

\maketitle

\section{Introduction}

Generalized K\"ahler geometry and generalized Calabi-Yau structures arose
from research on supersymmetric sigma models \cite{Gates}.  They were
rediscovered by Hitchin \cite{Hitchin}, growing out of investigations into
natural 
volume functionals on differential forms.  These points of view were connected
in the thesis of Gualtieri \cite{Gualtieri}.  These
structures have recently
attracted enormous interest in both the physics and mathematical communities as
natural generalizations of K\"ahler Calabi-Yau structures, inheriting a rich
physical and geometric theory.  The author and Tian \cite{STGK} developed a
natural notion of Ricci flow in generalized K\"ahler
geometry, and we will call this
flow \emph{generalized K\"ahler-Ricci flow} (GKRF).  
Explicitly it takes the form
\begin{gather} \label{GKRF}
 \begin{split}
 \dt g = - 2 \Rc^g + \frac{1}{2} \HH, \qquad \dt H = \gD_d H,\\
\dt I = L_{\theta_I^{\sharp}} I, \qquad \dt J = L_{\theta_J^{\sharp}} J,
\end{split}
\end{gather}
where $\HH_{ij} = H_{ipq} H_j^{pq}$, and $\theta_I, \theta_J$ are the Lee forms
of the corresponding Hermitian structures.

A special case of this flow arises when $[J_A,J_B] = 0$, a condition preserved
by the flow \cite{SPCFSTB}, and moreover causes the complex structures to be
fixed along the flow.  As shown in \cite{SPCFSTB}, the GKRF reduces to a single parabolic scalar
PDE in this setting.  We recall that,
suppressing all background geometry terms, the K\"ahler-Ricci flow is known to
reduce locally to the parabolic complex Monge-Ampere equation.  In the present
setting, the local reduction is to the parabolic complex ``twisted''
Monge-Ampere equation.  Namely, one has a splitting $\mathbb C^n = \mathbb C^k
\times \mathbb C^l$, and we denote $z = (z_1,\dots,z_n) = (z_+, z_-)$ where $z_+
\in \mathbb C^k, z_- \in \mathbb C^l$.  Then the ``twisted''\footnote{The
terminology `twisted Monge Ampere' appears in other places in the literature
often referring to a usual Monge Ampere equation modified by some further terms
involving specialized background geometry.  Despite this clash we will use this
terminology as it seems to economically capture the notion that the equation
exploits a nonstandard combination of Monge-Ampere operators.} equation is
\begin{align*}
\dt u =&\ \log \frac{ \det \i \del_+ \delb_+ u}{\det (- \i \del_- \delb_- u)}.
\end{align*}
As observed in \cite{SW}, this equation is formally related to the parabolic
complex Monge Ampere equation via partial Legendre transformation in the $z_-$
variables.  This observation was exploited to establish a $C^{2,\ga}$ estimate
of Evans-Krylov type for this equation, overcoming the nonconvexity which
prevents applying standard machinery.  This estimate can be combined with
further global a priori estimates which hold in specific geometric/topological
situations \cite{SPCFSTB,SBIPCF} to establish global existence and convergence
results for the GKRF.

Despite the partial results and natural estimates which have been
established for this flow, a full regularity theory is lacking due to the lack
of general a priori estimates on the parabolicity of the equation.  Here again
the
nonconvexity of the equation causes difficulty  as the potential function alone
cannot be added to test functions to apply the maximum principle as in the
traditional Monge-Ampere theory \cite{Yau}.  For this reason it is natural to pursue
alternative methods
for establishing low order estimates, and here we look to viscosity theory.  Our main
result establishes the existence of such solutions.  In the statement below
$\tau^*$ is the maximal possible smooth existence time based on cohomological
obstructions (see Definition \ref{GKtaustardef}).

\begin{thm} \label{mainthm} Let $(M^{2n}, g, J_A, J_B)$ be a generalized
K\"ahler manifold with $[J_A,J_B] = 0$.  There exists a unique maximal viscosity
solution to
GKRF on $[0,\tau^*)$, realized as the supremum of all subsolutions.
\end{thm}

\begin{rmk} 
\begin{enumerate}
\item In the work \cite{AFS} a general theory of viscosity solutions is
developed for equations on Riemannian manifolds.  They require adapting the
``variable-doubling method'' \emph{globally} on $M$,
which forces
the use of the global distance function.  The analytic details require some
convexity properties for the distance function, which are only satisfied under
strong curvature hypotheses such as nonnegative sectional curvature. 
\item A remarkable feature of the viscosity theory for the complex Monge Ampere
equation is that the traditional definition of viscosity subsolution naturally
picks out \emph{elliptic} subsolutions (cf. \cite{CMAV} Proposition 1.3)  For
instance, the function $\brs{z_1}^2 - \brs{z_2}^2 -  \brs{z_3}^2$ is \emph{not}
a viscosity solution of $\det \i \del\delb u = 1$ on $\mathbb C^3$.  This is
related to the simple but important observation that the supremum of two
subsolutions is again a subsolution.  This presents an extra challenge due to
the natural mixed plurisub/superharmonic condition needed for ellipticity of the
twisted equation, which for instance is not preserved under taking supremums. 
These issues are overcome by making careful definitions of sub/supersolutions
which naturally split up the two pieces of the ellipticity condition so that
part is satisfied by subsolutions, part by supersolutions.
\item While it is satisfying to construct
a global solution with some (very weak) regularity, it is of course unsatisfying
because ultimately we expect the solution to be smooth, and it is unclear if the
viscosity approach can eventually lead to the full regularity.  Viscosity theory
holds the promise to understand generalized K\"ahler-Ricci flow, perhaps for instance flowing through singularities.  This is the
approach taken in a series of works based on \cite{CMAV, CMAF2} in efforts to
better understand the complex Monge-Ampere equation/K\"ahler-Ricci flow in singular settings. 
In the course of the author's investigations into these works a crucial error
was discovered which renders those works and many subsequent works incomplete. 
This is explained in \S \ref{globalsec}.  Despite these errors in the 
proofs it still seems likely that the statements ultimately are true, although
we were unsuccessful in attempting to repair the existing approach.
\end{enumerate}
\end{rmk}

\subsection*{Acknowledgements:} The author would like to thank Patrick Guidotti
for helpful conversations on viscosity theory.

\section{Smooth Twisted Monge Ampere Flows} \label{background}

In this section we recall and refine the discussion in \cite{SPCFSTB} wherein
the
pluriclosed flow in the setting of
generalized K\"ahler geometry with commuting complex structures is reduced to a
fully nonlinear parabolic PDE.  First we recall the fundamental aspects of the
relevant differential geometry.

\subsection{Tangent bundle splitting} 

Let $(M^{2n}, g, J_A, J_B)$ be a generalized K\"ahler manifold satisfying
$[J_A,J_B] = 0$.  Define
\begin{align*}
 \Pi := J_A J_B \in \End(TM).
\end{align*}
It follows that $\Pi^2 = \Id$, and $\Pi$ is $g$-orthogonal, hence $\Pi$ defines
a $g$-orthogonal decomposition into its $\pm 1$ eigenspaces, which we denote
\begin{align*}
 TM = T_+ M \oplus T_- M.
\end{align*}
Moreover, on the complex manifold $(M^{2n}, J_A)$ we can similarly decompose the
complexified tangent bundle $T_{\mathbb C}^{1,0}$.  For notational simplicity we
denote
\begin{align*}
 T_{\pm}^{1,0} := \ker \left( \Pi \mp I \right) : T^{1,0}_{\mathbb C}
(M, J_A) \to T^{1,0}_{\mathbb C} (M, J_A).
\end{align*}
We use similar notation to denote the pieces of the complex cotangent bundle. 
Other tensor bundles inherit similar decompositions.  The one of most importance
to us is
\begin{align*}
 \Lambda^{1,1}_{\mathbb C}(M, J_A) =&\ \left(\Lambda^{1,0}_+ \oplus
\Lambda^{1,0}_- \right) \wedge \left(\Lambda^{0,1}_+ \oplus \Lambda^{0,1}_-
\right)\\
=&\ \left[\Lambda^{1,0}_+ \wedge \Lambda_+^{0,1} \right] \oplus
\left[\Lambda^{1,0}_+ \wedge \Lambda_-^{0,1} \right] \oplus
\left[\Lambda^{1,0}_- \wedge \Lambda_+^{0,1} \right] \oplus
\left[\Lambda^{1,0}_- \wedge \Lambda_-^{0,1} \right].
\end{align*}
Given $\mu \in \Lambda^{1,1}_{\mathbb C}(M, J_A)$ we will denote this
decomposition as
\begin{align} \label{oneoneproj}
 \mu := \mu^+ + \mu^{\pm} + \mu^{\mp} + \mu^-.
\end{align}
These decompositions allow us to decompose differential operators as well.  In
particular we can express 
\begin{align*}
d = d_+ + d_-, \qquad \del = \del_+ + \del_-, \qquad \delb = \delb_+ + \delb_-.
\end{align*}
The crucial differential operator governing the local generality of generalized
K\"ahler metrics in this setting is
\begin{align*}
\square := \i \left( \del_+ \delb_+ - \del_- \delb_- \right).
\end{align*}

\subsection{A characteristic class}

\begin{defn} \label{chidef} Let $(M^{2n}, J_{A}, J_B)$ be a bicomplex manifold
such that $[J_A,J_B] = 0$.  Let
\begin{align*}
\chi(J_A,J_B) = c^+_1(T^{1,0}_+) - c^-_1(T^{1,0}_{+}) +
c_1^-(T^{1,0}_{-}) - c_1^+(T^{1,0}_{-}).
\end{align*}
The meaning of this formula is the following: fix Hermitian metrics $h_{\pm}$ on
the holomorphic line bundles $\det T^{1,0}_{\pm}$, and use these to define
elements of $c_1(T^{1,0}_{\pm})$, and then project according to the
decomposition (\ref{oneoneproj}).  In particular, given such metrics $h_{\pm}$
we let $\rho(h_{\pm})$ denote the associated representatives of
$c_1(T_{\pm}^{1,0})$, and then let
\begin{align*}
 \chi(h_{\pm}) = \rho^+(h_+) - \rho^{-}(h_+) + \rho^-(h_-) - \rho^+(h_-).
\end{align*}
This definition yields a well-defined class
in a certain cohomology group, defined in \cite{SPCFSTB}, which we now describe.
\end{defn}

\begin{defn} Let $(M^{2n}, J_A, J_B)$ be a bicomplex manifold with $[J_A,J_B]
= 0$.  Given $\zeta_A \in \Lambda^{1,1}_{J_A,\mathbb R}$, let $\zeta_B = -
\zeta_A(
\Pi
\cdot, \cdot) \in \Lambda^{1,1}_{J_B,\mathbb R}$.  We say that $\zeta_A$ is
\emph{formally generalized K\"ahler} if
\begin{gather} \label{FGK}
\begin{split}
d^c_{J_A} \zeta_A =&\ - d^c_{J_B} \zeta_B,\\
d d^c_{J_A} \zeta_A =&\ 0.
\end{split}
\end{gather}
\end{defn}

This definition captures every aspect of a generalized K\"ahler metric compatible with $J_A, J_B$, except for being positive definite.  As we will show in Lemma \ref{localddbar} below, such forms are locally expressed as $\square f$.  It is therefore natural to define the following cohomology space.

\begin{defn} Let $(M^{2n}, J_A, J_B)$ denote a bicomplex manifold
such that $[J_A,J_B] = 0$.  Let
\begin{align*}
H^{1,1}_{GK} := \frac{ \left\{ \zeta_A \in \Lambda^{1,1}_{J_A, \mathbb R}\ |  \
\zeta_A 
\mbox{ satisfies } (\ref{FGK}) \right\}}{ \left\{ \square f \ | \ f \in
C^{\infty}(M) \right\}}.
\end{align*}
\end{defn}

It follows from direct calculations using the transgression formula for $c_1$ (cf. \cite{SPCFSTB}) that $\chi$ yields a well-defined class in $H^{1,1}_{GK}$.

\subsection{Pluriclosed flow in commuting generalized K\"ahler geometry}

With this setup we describe how to reduce pluriclosed flow to a scalar PDE in
the
setting of commuting generalized K\"ahler manifolds.  First we recall that it
follows from (\cite{SPCFSTB} Proposition 3.2, Lemma 3.4)
that the pluriclosed flow in this setting reduces to
\begin{align} \label{PCFGK}
 \dt \gw =&\ - \chi(\gw_{\pm}).
\end{align}
To capture the idea of the formal maximal existence time, we first define the analogous notion to the
``K\"ahler cone,'' which we refer to as $\mathcal P$, the ``positive cone:"
\begin{defn}
 Let $(M^{2n}, g, J_A, J_B)$ denote a bicomplex manifold
such that $[J_A,J_B] = 0$.  Let
\begin{align*}
 \mathcal P := \left\{ [\zeta] \in H^{1,1}_{GK}\ |\ \exists\ \gw \in [\zeta],
\gw > 0
\right\}.
\end{align*}
\end{defn}
From the discussion above, we thus see that a solution to (\ref{PCFGK})
induces a solution to an ODE in $\mathcal P$, namely
\begin{align*}
 [\gw_t] = [\gw_0] - t \chi.
\end{align*}
It is clear now that there is a formal obstruction to the maximal smooth
existence time of the flow.
this setting.
\begin{defn} \label{GKtaustardef} Given $(M^{2n}, g, J_A, J_B)$ a generalized
K\"ahler manifold with $[J_A,J_B] = 0$, let
\begin{align*}
 \tau^*(g) := \sup \left\{ t \geq 0\ |\ [\gw] - t \chi \in \mathcal P \right\}.
\end{align*}
\end{defn}
Now fix $\tau < \tau^*$, so that by hypothesis if we fix arbitrary metrics
$\til{h}_{\pm}$ on $T^{1,0}_{\pm}$, there exists $a \in C^{\infty}(M)$ such that
\begin{align*}
 \gw_0 - \tau \chi(\til{h}_{\pm}) + \square a > 0.
\end{align*}
Now set $h_{\pm} = e^{\pm \frac{a}{2 \tau}} \til{h}_{\pm}$.  Thus $\gw_0 - \tau
\chi(h_{\pm}) > 0$, and by convexity it follows that
\begin{align*}
 \hat{\gw}_t := \gw_0 - t \chi(h_{\pm}) > 0
\end{align*}
is a smooth one-parameter family of generalized K\"ahler metrics.  Furthermore,
given a function $f \in C^{\infty}(M)$, let
\begin{align*}
 \gw^f := \hat{\gw} + \square f,
\end{align*}
with $g^f$ the associated Hermitian metric.  Now suppose that $u$ satisfies
\begin{align} \label{scalarPCF2}
 \dt u =&\ \log \frac{ (\gw^u_+)^k \wedge (\zeta_-)^l}{ (\zeta_+)^l \wedge
(\gw^u_-)^l},
\end{align}
where $\zeta$ denotes the K\"ahler form of the Hermitian metric $h$.  An
elementary calculation using the transgression formula for the first Chern class
(\cite{SPCFSTB} Lemma 3.4) yields that $\gw_u$ solves (\ref{PCFGK}).

\subsection{Twisted Monge-Ampere flows}

We now codify the discussion of the previous subsection by making some general
definitions, and then use these to define our notion of viscosity
sub/supersolutions.

\begin{defn} \label{TMAdef} Let $(M^{2n}, g, J_A, J_B)$ be a generalized
K\"ahler manifold with $[J_A, J_B] = 0$.  Fix
\begin{enumerate}
\item $\gw$ a continuous family of formally generalized K\"ahler  forms.
\item $0 \leq \mu_+(z,t) \in C^0(M, \gL_+^{k,k}),  0 \leq \mu_-(z,t) \in C^0(M,
\gL_-^{l,l})$ continuous families of partial volume forms.
\item $F : M \times [0,T) \times \to \mathbb R$ a continuous function.
\end{enumerate}
A function $u \in C^2(M \times [0,T))$ is a solution of \emph{$(\gw,\mu_{\pm},
F)$-twisted Monge Ampere flow} if
\begin{enumerate}
\item $\gw_{u_t} > 0$ for all $t \in [0,T)$.
\item $(\gw_+ + \i \del_+ \delb_+ u)^k \wedge \mu_- = e^{u_t + F(x,t)} (\gw_- -
\i \del_- \delb_- u)^l \wedge \mu_+$.
\end{enumerate}
\end{defn}

\begin{defn} \label{VTMAdef} Let $(M^{2n}, g, J_A, J_B)$ be a generalized
K\"ahler manifold with $[J_A, J_B] = 0$.  Fix data $(\gw,\mu_{\pm},F)$ as in
Definition \ref{TMAdef}.
A function $u \in \USC(M \times [0,T))$ is a \emph{viscosity subsolution of
$(\gw,\mu_{\pm}, F)$-twisted Monge Ampere flow} if
for all $\phi \in C^{\infty}(M \times [0,T))$ such that $u -
\phi$ has a local
maximum at $(z,t) \in M \times (0,T)$, one has that, at $(z,t)$,
\begin{align*}
(\gw_+ + \i \del_+ \delb_+ \phi)^k \wedge \mu_- \geq&\ e^{\phi_t + F(x,t)}
\left[ (\gw_- - \i \del_- \delb_-
\phi)_+ \right]^l \wedge \mu_+,
\end{align*}
where for a section $\zeta \in \Lambda^{1,1}_-$ the notation $\zeta_+^k$
means $\zeta^l$ if $\eta > 0$ and zero otherwise.

Likewise, a function $v \in \LSC(M \times [0,T))$ is a \emph{viscosity
supersolution of $(\gw,\mu_{\pm}, F)$-twisted Monge Ampere flow}
if for all $\phi \in C^{\infty}(M \times [0,T))$ such that $v -
\phi$ has a local
minimum at $(z,t) \in M \times (0,T)$, one has that, at $(z,t)$,
\begin{align*}
\left[ (\gw_+ + \i \del_+ \delb_+ \phi)_+ \right]^k \wedge \mu_- \leq&\ e^{\phi_t + F(x,t)}
\left(\gw_- - \i \del_-
\delb_- \phi \right)^l \wedge \mu_+,
\end{align*}
where for a section $\eta \in \Lambda^{1,1}_+$ the notation $\eta_+^k$
means $\eta^k$ if $\eta > 0$ and zero otherwise.
\end{defn}

A remarkable feature of the viscosity
theory for complex
Monge-Ampere equations is that it naturally selects elliptic solutions to the
problem.  In a sense it is forced upon the solutions through the use of the
projection operators onto the positive part of the complex Hessian of the test
functions, and the fact that the inequality must hold for arbitrary test
functions, as explained in (\cite{CMAV} Proposition 1.3).  In
our case the notion of ellipticity is more delicate, and
yet the viscosity theory still allows us to set up our definitions so as to
ensure we obtain elliptic solutions to the problem.  This is surprising due to
the nonconvexity of the equation at hand.  

Even further, the Perron process,
which involves
taking supremums of subsolutions, natually preserves the plurisubharmonicity
of subsolutions in the $z_+$ directions, but would not preserve the
plurisuperharmonicity in the $z_-$ directions if we
attempted to impose this by hand.  Only a fortiori, having constructed a
sub/supersolution at the end of the Perron process, do we ensure that our final
solution is parabolic.  We clarify this in the rest of the subsection.  The first step is to exhibit a local version of the $\del\delb$-lemma adapted to this setting.  This result is stated in
\cite{Gates} without proof, which is however elementary.

\begin{lemma} \label{localddbar} Let $\gw = \gw_+ + \gw_-$ be formally
generalized K\"ahler on $U
\subset \mathbb C^k \times \mathbb C^l$.  There exists $f \in C^{\infty}(U)$
such that $\gw = \square f$.
\begin{proof} First observe that since $d_+ \gw_+ = 0$, on each $w \equiv
\mbox{const}$ complex $k$-plane we can apply the $\del\delb$-lemma to obtain a
function $\psi_+(z)$ such that $\i \del_+ \delb_+ \psi_+ = \gw_+$ on that plane.
Since $\gw_+$ is smooth, we can moreover choose these on each slice so that the
resulting function
$\psi_+(z,w)$ is smooth, and satisfies $\i \del_+ \delb_+ \psi_+ = \gw_+$ on
$U$. 
Arguing similarly we obtain a function $\psi_-$ such that $\i \del_- \delb_-
\psi_- = \gw_-$ everywhere on $U$.  We note now that the fact that $\gw$ is
pluriclosed implies that
\begin{align*}
0 =&\ \i \del_+ \delb_+ \gw_- + \i \del_- \delb_- \gw_+ = - \del_+ \delb_+
\del_- \delb_- \left(\psi_+ + \psi_- \right).
\end{align*}
We next claim that any element in the kernel of the operator $\del_+ \delb_+
\del_- \delb_-$, in particular $\psi_+ + \psi_-$, can be expressed as
\begin{align} \label{localddbar10}
\psi_+ + \psi_- =&\ \gl_1(z,\bz,w) + \bar{\gl}_1(z,\bz,\bw) + \gl_2(w,\bw,z) +
\bar{\gl}_2(w,\bw,\bz).
\end{align}
To see this we first note that if $\phi := \psi_+ + \psi_-$ satisfies $\del_+
\delb_+ \del_- \delb_- \phi = 0$, then $\del_- \delb_- \phi$ can be expressed as
the
real part of a $\delb_+$-holomorphic function, so $(\del_- \delb_- \phi)_{\gw_i
\bw_j}
 = \mu^{i\bj}_1(w,\bw,z) + \bar{\mu}^{i\bj}_1(w,\bw,\bar{z})$, where the
indices on the $\mu$ refer to the fact that each component of the
$\del_-\delb_-$-Hessian can be expressed this way.  It follows that $\gD_- \phi
:=
\i \phi_{,w_i \bw_i}$ is the real part of a $\delb_+$-holomorphic function. 
Applying
the Green's function on each $z$-slice it follows that $\phi$ can be expressed
as
the real part of a $\delb_+$-holomorphic function, up to the addition of an
arbitrary $\delb_-$-holomorphic function.  Thus (\ref{localddbar10}) follows.

We claim that $f = \psi_+ - \gl_2 - \bar{\gl}_2$ is the required potential
function.  In particular, since $\i \del_+ \delb_+ \left( \gl_2 + \bar{\gl}_2
\right) = 0$ it follows that $\i \del_+ \delb_+ f = \gw_+$.  Also, we compute
using (\ref{localddbar10}),
\begin{align*}
- \i \del_- \delb_- f =&\ - \i \del_- \delb_- \left( - \psi_- + \gl_1 +
\bar{\gl}_1 \right)\\
=&\ \i \del_- \delb_- \psi_-\\
=&\ \gw_-.
\end{align*}
The lemma follows.
\end{proof}
\end{lemma}

\begin{lemma} \label{goodcover} Let $(M^{2n}, g,
J_A, J_B)$ be a generalized
K\"ahler manifold with $[J_A, J_B] = 0$.  Suppose $\gw_t, t \in [0,T]$ is a
one-parameter family of smooth generalized K\"ahler metrics on $M$.  There
exists a locally finite open cover $\mathcal U = \{U_{\gb}\}$ of $M$ such that
\begin{enumerate}
 \item Each $U_{\gb}$ is the domain of a bicomplex coordinate chart.
 \item For each $\gb$ there is a smooth function $f_{\gb} : U_{\gb} \times [0,T]
\to \mathbb R$ such that $\gw = \square f$.
\end{enumerate}
\begin{proof} The existence of local bicomplex coordinates around each point
follows from (\cite{ApostGualt} Theorem 4), and then the existence of a
locally finite cover follows from standard arguments.  At any time $t$ we can
construct a local potential $f$ by Lemma \ref{localddbar}, and it is clear by
the proof of that Lemma that $f$ can be chosen to depend smoothly on $\gw$, and
so the lemma follows.
\end{proof}
\end{lemma}

\begin{lemma} \label{localizedflow}
Let $(M^{2n},
g, J_A, J_B)$ be a generalized
K\"ahler manifold with $[J_A, J_B] = 0$.  Fix data $(\gw,\mu_{\pm},F)$ as in
Definition \ref{TMAdef}, and fix a cover $\mathcal U$ as in Lemma
\ref{goodcover}.  Suppose that on $U_{\gb} \in \mathcal U$ there are continuous
density functions $\zeta_{\pm}$ satisfying
\begin{align*}
\mu_+ =&\ e^{\zeta_+} (\i dz_+^1 \wedge d \bar{z}_+^1) \wedge \dots \wedge (\i
dz_+^k \wedge d \bar{z}_+^k)\\
\mu_- =&\ e^{\zeta_-} (\i dz_-^1 \wedge d \bar{z}_-^1) \wedge \dots \wedge (\i
dz_-^k \wedge d \bar{z}_-^k).
\end{align*}
If $u$ is a subsolution of $(\gw, \mu_{\pm}, F)$-twisted Monge Ampere flow, then
$u_{\gb} := u + f_{\gb}$ is a subsolution of
\begin{align*}
(\i \del_+\delb_+ w)^k \geq e^{w_t - f_t + F(x,t) + \zeta_+ - \zeta_-} (- \i \del_-
\delb_- w)_+^l.
\end{align*}
Likewise, if $v$ is a viscosity supersolution of $(\gw, \mu_{\pm}, F)$-twisted
Monge Ampere flow, then $v_{\gb} := v + f_{\gb}$ is a supersolution of
\begin{align*}
(\i \del_+\delb_+ w)_+^k \leq e^{w_t - f_t + F(x,t) + \zeta_+ - \zeta_-} (- \i \del_-
\delb_- w)^l.
\end{align*}
\begin{proof} This is an immediate consequence of unraveling the definitions.
\end{proof}
\end{lemma}

Observe that the inequalities defining
sub/supersolutions in Lemma \ref{localizedflow} are expressed as inequalities of
scalars in the chosen coordinates, whereas the original inequalities of
Definition \ref{VTMAdef} are expressed in terms of sections of $\Lambda^{n,n}$. 
Moreover, the meaning of viscosity sub/supersolution in this context is the
classic one.  As the key arguments in the proofs of the comparison theorems are
local in nature, it suffices to consider this localized version of the flow,
which has the advantage of stripping away much notation and making things more
concrete in coordinates.  We will refer to this setup informally as a
\emph{localized flow}.  Now we are ready to state our ellipticity claim.

\begin{lemma} \label{ellipticitylemma} Local viscosity subsolutions of twisted
Monge-Ampere flow as in Lemma \ref{localizedflow} are plurisubharmonic in the
$z_+$-variables, and
viscosity supersolutions of twisted Monge-Ampere flow as in Lemma
\ref{localizedflow} are plurisuperharmonic in the $z_-$-variables.
\begin{proof} Let $u$ be a local viscosity subsolution of twisted Monge-Ampere
flow.  Without loss of generality we assume the domain is $B_1(0) \times [0,T)$.
 Fix $(z_0,t_0) \in B_1(0) \times [0,T)$ such that $u(z_0)
\neq - \infty$.  Choose a function $\phi \in C^2(B_1(0) \times [0,T))$ such that
$u -
\phi$ has a local maximum at $(z_0,t_0)$.  It follows directly from the
definition of subsolution that 
\begin{align*}
(\i \del_+ \delb_+ \phi)^k \geq e^{\phi_t + F(z_0,t_0)} (- \i \del_- \delb_-
\phi)_+^l \geq 0.
\end{align*}
We claim that $\i \del_+ \delb_+ \phi \geq 0$.  First note that, if we fix a $k
\times k$ Hermitian
positive semidefinite matrix $H_+$, and set
\begin{align*}
 \phi_{H_+}(z,t) := \phi(z,t) + H_+(z_+ - (z_0)_+)(\bar{z}_+ - (\bar{z}_0)_+).
\end{align*}
The function $\phi_{H_+}$ has a local maximum at $(z_0,t_0)$ as well.  Hence
arguing as above we have
\begin{align*}
(\i \del_+ \delb_+ \phi_{H_+})^k = (\i \del_+ \delb_+ \phi + H_+)^k\geq 0.
\end{align*}
Since $H_+$ is arbitrary, by an elementary linear algebra argument this implies
$\i \del _+ \delb_+ \phi \geq 0$.  It then follows that for any positive
definite matrix $H_+$ one has
\begin{align*}
 H_+^{\bj_+ i_+} \frac{\del^2 \phi}{\del z^i \del \bar{z}^j} \geq 0.
\end{align*}
This implies that $u$ is a viscosity subsolution of $\gD_H \phi \geq 0$. Since
$H_+$ is arbitrary, using results from linear elliptic PDE theory
(\cite{Hormander}) as in (\cite{CMAV} Proposition 1.3) it follows that $u$ is
plurisubharmonic in the $z_+$-variables.  The argument for $u$ being
plurisuperharmonic in the $z_-$-variables is directly analogous.
\end{proof}
\end{lemma}

\section{Proof of Theorem \ref{mainthm}}

\begin{lemma} \label{weakcomparison} Let $(M^{2n}, g, J_A, J_B)$ be a generalized K\"ahler manifold
with $[J_A,J_B] = 0$.  Fix data $(\gw,\mu_{\pm},F)$ as in Definition
\ref{TMAdef}.  Suppose $\ubar{u}$ is a bounded viscosity subsolution of
$(\gw,\mu_{\pm},F)$-twisted Monge-Ampere flow, and suppose $\bar{u}$ is a smooth
supersolution of $(\gw,\mu_{\pm},F)$-twisted Monge-Ampere flow.  If $\bar{u}(x,0)
\geq \ubar{u}(x,0)$ for all $x \in M$, then $\bar{u}(x,t) \geq \ubar{u}(x,t)$ for all $(x,t)
\in M \times [0,T)$.
\begin{proof} Suppose there exists $(x_0,t_0) \in M \times [0,T)$ such that
$\bar{u}(x_0,t_0) < \ubar{u}(x_0,t_0)$.  It follows directly from the definitions that
for $\gd > 0$, the function
\begin{align*}
\bar{u}_{\gd}(x,t) := \bar{u}(x,t) + \frac{\gd}{T - t}
\end{align*}
is also a smooth supersolution.  Moreover, for $\gd$ chosen sufficiently small
it follows that $\bar{u}_{\gd}(x_0,t_0) < \ubar{u}(x_0,t_0)$.  Since $\ubar{u}$ is bounded and
$\lim_{t \to T} u_{\gd}^*(x,t) = \infty$ for all $x \in M$, it follows that $\ubar{u}
- \bar{u}_{\gd}$ attains a positive maximum at some point $(x_0',t_0')$, $0 < t_0' <
T$.

The function $u_{\gd}^*$ is smooth, and so can be used in the definition of
$\ubar{u}$ being a subsolution to yield, at the point $(x_0',t_0')$, the inequality
\begin{align*}
(\gw_+ + \i \del_+ \delb_+ u_{\gd}^*)^k \wedge \mu_- \geq e^{(\bar{u}_{\gd})_t +
F(x_0',t_0')} (\gw_- - \i \del_- \delb_- \bar{u}_{\gd})_+^l \wedge \mu_+.
\end{align*}
Since $\bar{u}(\cdot,t) \in \PP_{\gw_t}$ for all $t$, we can ignore the projection
operator on the right hand side and apply elementary identities to obtain
\begin{align*}
(\gw_+ + \i \del_+ \delb_+ \bar{u})^k \wedge \mu_- \geq e^{(\bar{u})_t + \gd/(T -
t_0')^2 + F(x_0',t_0')} (\gw_- - \i \del_- \delb_- \bar{u})^l \wedge \mu_+.
\end{align*}
On the other hand, since $\bar{u}$ is already a supersolution and $\bar{u}(\cdot,t) \in
\PP_{\gd_t}$ for all $t$ we have
\begin{align*}
(\gw_+ + \i \del_+ \delb_+ \bar{u})^k \wedge \mu_- \leq e^{\bar{u}_t + F(x_0',t_0')}
(\gw_- - \i \del_- \delb_- \bar{u})^l \wedge \mu_+.
\end{align*}
Putting the previous two inequalities together yields
\begin{align*}
e^{\gd/(T - t_0')^2} \leq e^{-u_t^* -  F(x_0',t_0')} \frac{(\gw_+ + \i \del_+
\delb_+ \bar{u})^k \wedge \mu_-}{(\gw_- - \i \del_- \delb_- \bar{u})_+^l \wedge \mu_+}
\leq 1,
\end{align*}
a contradiction.
\end{proof}
\end{lemma}

\begin{proof}[Proof of Theorem \ref{mainthm}] 

We first observe the existence of smooth, bounded sub/supersolutions.  In
particular,
since $g^{u_0}, h$ are smooth metrics, one has that
\begin{align*}
\sup_M \brs{\log \frac{\det g_+^{u_0} \det h_-}{\det h_+ \det g_-^{u_0}}} \leq A.
\end{align*}
It follows immediatley that the smooth functions
\begin{align*}
\bar{u} := u_0 + t A, \qquad \ubar{u} := u - t A,
\end{align*}
are smooth sub/supersolutions to the problem.

Now let
\begin{align*}
u = \sup\{ w\ |\ \ubar{u} \leq w \leq \bar{u}, \ w \mbox{ is a subsolution to } (\gw,
\mu_{\pm},F) \mbox{ twisted MA flow} \}.
\end{align*}
We claim that $u$ is a viscosity solution in the sense that the usc regularization $u^*$ is a subsolution, whereas the lower semicontinuous regularization $u_*$ is a supersolution.  It is a standard argument (cf. \cite{Users}) to show that, as a supremum of subsolutions, $u$ itself is a
subsolution to $(\gw,\mu_{\pm},F)$-twisted Monge Ampere flow.  It follows that in fact $u^* = u$ is a subsolution.  Next we show that $u_*$ is a
supersolution.  If not, there exists $(z_0,t_0) \in M \times [0,T)$ and $\phi$ a
$C^2$ function such that $u_* - \phi$ has a local minimum of zero at $(z_0,t_0)$
and, at that point,
\begin{align*}
\left[ (\gw_+ + \i \del_+ \delb_+ \phi)_+ \right]^k \wedge \mu_- > e^{\phi_t + F(x,t)}
\left(\gw_- - \i \del_-
\delb_- \phi \right)^l \wedge \mu_+.
\end{align*}
Choose coordinates around $z_0$, fix constants $\gg,\gd > 0$ and consider
\begin{align*}
\phi_{\gg,\gd} = \phi + \gd - \gg \brs{z}^2.
\end{align*}
It follows that 
\begin{align*}
\left[ (\gw_+ + \i \del_+ \delb_+ \phi_{\gg,\gd})_+ \right]^k \wedge \mu_- > e^{\phi_t + F(x,t)}
\left(\gw_- - \i \del_-
\delb_- \phi_{\gg,\gd} \right)^l \wedge \mu_+
\end{align*}
on $P_r(z_0)$, for sufficiently small $r > 0$.  If we choose $\gd = (\gg r^2)/8$
then it follows that $u_* > \phi_{\gg,\gd}$ for $r/2 \leq \brs{\brs{z}} \leq r$,
whereas $\phi_{\gg,\gd}(z_0,t_0) > u_*(z_0,t_0) + \gd$.  We now define, supressing
the identification with the given coordinate chart,
\begin{align*}
\Phi(z,t) = \begin{cases}
\max \{ u_*(z,t), \phi_{\gg,\gd}(z,t) \} &\ z \in B_r(z_0)\\
u_*(z,t) &\ \mbox{otherwise}
\end{cases}.
\end{align*}
As the supremum of subsolutions, $\Phi$ is a subsolution to
$(\gw,\mu_{\pm},F)$-twisted Monge Ampere flow.  Now choose a sequence $(z_n,t_n)
\to (z_0,t_0)$ such that $u(z_n,t_n) \to u_*(z_n,t_n)$.  For sufficiently large
$n$ it follows that $\Phi(z_n,t_n) = \phi_{\gg,\gd}(z_n,t_n) > u(z_n,t_n)$,
contradicting the definition of $u$.
\end{proof}

\section{Global comparison principle for singular equations} \label{globalsec}
\subsection{Localization and comparison principles}

The viscosity solution constructed in Theorem \ref{mainthm} is sometimes in the literature referred to as a ``Perron discontinuous viscosity solution."  The function constructed is not even known to be continuous.  Typically what is required to show this is a more general comparison principle, generalizing Lemma \ref{weakcomparison} to the case of an arbitrary (not just smooth) supersolution.  In the case of fully nonlinear second order equations on domains in $R^n$, this is achieved by the ``Jensen-Ishii maximum principle,'' a delicate technique exploiting various properties of $\mathbb R^n$ in an essential way.

While it seems natural that these ideas should extend to the case of equations on manifolds, there seem to be subtle technical issues in making the Jensen-Ishii method work.  While these have been overcome assuming background curvature conditions in \cite{AFS}, recent efforts to overcome these obstacles in the case of the complex Monge Ampere equation appear to be incomplete.  In particular, in this subsection we describe a crucial error in the paper
\cite{CMAV}.  The main results in \cite{CMAV} claim to use ideas from viscosity
theory to establish estimates/existence results for the complex Monge Ampere
equation with singular background measures.  Prior work using the viscosity
method for equations on compact manifolds appeared in \cite{AFS}, where the
background geometry appears explicitly in the argument and as such requires a
condition on the curvature for the method to succeed.  The paper \cite{CMAV}
claims
to get around this by explicitly localizing the proof using cutoff functions and
the pure scalar form of the local PDE.  As discussed above, the
central tool required is a comparison principle for viscosity
sub/supersolutions.  This is Theorem 2.14 in
\cite{CMAV}, and contains a key logical gap the author was unable to repair.

The proof (\cite{CMAV} pages 17-19) mostly follows standard lines which I will
briefly describe, calling attention to the one specific false claim, and then
describing the role it plays in the proof at large.  Call the sub/supersolutions
$w_*, w^*$ respectively.  There
is no boundary so one in general needs to show $w_* \leq w^*$.  The standard
method in viscosity theory is to use `variable doubling' with `penalization' and
consider the function
\begin{align*}
 \Phi_{\ga}(x,y) = w_*(x) - w^*(y) - \frac{1}{2} \ga \brs{x - y}^2.
\end{align*}
This is an upper semicontinuous function whose maximum occurs near the diagonal
for large $\ga$.
 The strategy of \cite{CMAV} in using this method on manifolds is to modify this
by further penalizing with a cutoff function.

In particular, they consider their equation on say a ball of radius $4$.  They
produce a smooth function $\phi_3: M \times M \to \mathbb R$ on page 18 which
satisfies:
\begin{enumerate}
 \item $\phi_3 \geq 0$
 \item $\phi_3^{-1}(0) = \gD \cap \{\phi_2 \leq \eta\}$
 \item $\phi_3|_{M^2 \backslash B(0,2)^2} > 3C$.
\end{enumerate}
Here $\gD$ is the diagonal and the function $\phi_2$ is an arbitrary smooth
function satisfying $\phi_2|_{B(0,1)^2} < -1$ and $\phi_2|_{M^2 \backslash
B(0,2)^2} > C$ for a large constant $C$.  Furthermore $1  >> \eta > 0$ is chosen
so that $-\eta$ is a regular value of $\phi_2$ and $\phi_2|_{\gD}$.  This is not
specified in further detail in \cite{CMAV}, but the key properties used are that
the function $\phi_3$
vanishes along part of the diagonal, say the part contained in $B(0,1)^2$, and
is large away from $B(0,2)^2$.  Now let
\begin{align*}
 \Phi_{\ga}(x,y) = w_*(x) - w^*(y) - \phi_3(x,y) - \frac{1}{2} \ga \brs{x -
y}^2,
\end{align*}
and choose a sequence $(x_{\ga},y_{\ga})$ realizing the supremum of
$\Phi_{\ga}$ as $\ga \to \infty$, and take a limit point $(\hat{x},\hat{y})$. 
A result from \cite{Users}, utilized as (\cite{CMAV} Lemma 2.15), yields
$\hat{x} = \hat{y}, \hat{x} \in \gD \cap \{\phi_2 \leq - \eta
\}$.  In other words, the limit point is on the diagonal, and in the zero set of
$\phi_3$.

Next the authors apply the `Jensen-Ishii maximum principle' (recorded as
\cite{CMAV} Lemma 2.16),
with the penalization function $\phi = \phi_3 + \frac{1}{2} \ga \brs{x - y}^2$
to obtain, for arbitrary $\ge > 0$, test jets $(p_*, X_*), (p^*, X^*)$
satisfying
\begin{align*}
 \left( 
 \begin{matrix}
X_* & 0\\
0 & - X^*
 \end{matrix} \right) \leq A + \ge A^2,
\end{align*}
where $A = D^2 \phi(x_{\ga},y_{\ga})$, i.e.
\begin{align*}
 A = \ga \left(
 \begin{matrix}
  I & - I\\
  - I & I
 \end{matrix} \right) + D^2 \phi_3(x_{\ga},y_{\ga})
\end{align*}
It is crucial to the rest of the proof that
one has $X^* \geq X_* \geq 0$.  In the purely local case where $\phi_3 = 0$ one
chooses $\ge$ appropriately relative to $\ga$ to obtain
\begin{align*}
 \left( 
 \begin{matrix}
X_* & 0\\
0 & - X^*
 \end{matrix} \right) \leq 3 \ga \left(
 \begin{matrix}
  I & - I\\
  - I & I
 \end{matrix} \right),
\end{align*}
which immediately implies the required inequality $X^* \geq X_*$ .  With
$\phi_3$ in place, it is necessary to show that its Hessian is not just small,
but decaying at the rate of $\ga^{-2n}$.  The authors correctly note that,
`...the Taylor series (of $\phi_3$) vanishes up to order $2n$ on $\gD \cap
\{\phi_2
\leq - \eta\}$'.  But then it is claimed that this implies
this implies
\begin{align*}
 D^2 \phi_3(x_{\ga},y_{\ga}) = O(d(x_{\ga},y_{\ga})^{2n}) = o(\ga^{-n}).
\end{align*}
The second equality is trivial since by construction
one easily has $d(x_{\ga},y_{\ga})^2 = o (\ga^{-1})$.  However, the first
equality is false.  This would be true if one could Taylor expand around the
point $(x_{\ga},x_{\ga})$, \emph{assuming $(x_{\ga},x_{\ga}) \in
\gD \cap \{\phi_2 \leq - \eta \}$, in other words, if $\phi_3(x_{\ga},x_{\ga}) =
0$}.  It is clear however that 
construction of $\Phi_{\ga}$ and the corresponding sequence $(x_{\ga},y_{\ga})$
allows for $\phi_3(x_{\ga},x_{\ga}) > 0, \phi(y_{\ga},y_{\ga}) > 0$.
 While it is true that the limit point $(\hat{x},\hat{x})$ satisfies
$\phi_3(\hat{x},\hat{x}) = 0$ it does not follow that it is true along the
sequence.  

This oversight concerns exactly the key difficulty in applying these classic
techniques on manifolds, which is how to `localize'.  Moreover, elementary
arguments
seem to show that any such cutoff function chosen for the role of $\phi_3$
produces an `error term' which cannot be overcome to conclude the crucial
inequality $X^* > 0$.  Hence the question of `is it possible to utilize the
Jensen-Ishii maximum principle on manifolds' remains largely open, other than
the work \cite{AFS}

\subsection{Outlook}

We have shown that the proof of the main comparison principle of \cite{CMAV},
Theorem 2.14, is flawed, rendering the proofs of all of the main results in that
paper incomplete.  One of the main claims of \cite{CMAV}, namely Theorem C which
asserts the continuity of the pluripotential-theoretic solution to some
degenerate Monge Ampere equations with right hand side in $L^p, p > 1$, has been
used in many further works, which now appear incomplete.  In particular,
\cite{BermanCOT} Theorem 1.3, relies on this result and so is incomplete.  We
note that \cite{Fang} presents two proofs of a $C^0$ estimate for the $J$-flow,
one of which relies on (\cite{CMAV} Theorem C), the other of which is direct and complete.  Theorem 5.4 of \cite{Bermangeod} relies on (\cite{CMAV} Theorem C),
although it is difficult to determine how central this is to the main arguments.
 In \cite{
DiGuedj} Proposition 1.4 the authors invoke the arguments of \cite{CMAV}
rendering this proposition, on which much of the paper is based, incomplete. 
Furthermore, the authors of \cite{CMAV} use (\cite{CMAV} Theorem C) to claim
continuous approximation of plurisubharmonic functions in \cite{EGZCont} (cf.
Main Theorem, Corollary, Theorem 2.3, and another invocation of the same flawed
localization technique on page 8).  This claim is used in a central way in the
recent papers (\cite{Darvas} \S 3.2, \cite{LuNguyen} Theorem 2).  

The same localization method has been employed to claim results on the
K\"ahler-Ricci flow as well.  First we note that the results of \cite{CMAF1} are
purely local in nature and do not feature this argument.  However, the flawed
localization result in used in the main comparison result of \cite{CMAF2},
Theorem 2.1, rendering the proofs of all of the main results in that paper
incomplete.   Also, in (\cite{EGZKRF} page 20) the same erroneous localization
technique is applied.

\end{document}